\documentclass[envcountsame,envcountsect]{svmult}

\usepackage{makeidx}         % allows index generation
\usepackage{graphicx}        % standard LaTeX graphics tool
\usepackage{multicol}        % used for the two-column index
\usepackage[bottom]{footmisc}% places footnotes at page bottom
\makeindex             % used for the subject index
\usepackage{amsmath,amsfonts,amssymb,latexsym,amscd}
\usepackage{enumerate}
\usepackage{xy}\xyoption{all}

\newcommand{\R}{\mathbb{R}}
\newcommand{\C}{\mathbb{C}} 

\newcommand{\Z}{\mathbb{Z}}

\newcommand{\Flags}{\mathcal B}
\newcommand{\Cell}{\mathcal C}
\newcommand{\BS}{\mathcal B\mathcal S}

\newcommand{\h}{{\mathfrak h}}
\newcommand{\n}{{\mathfrak n}}
\newcommand{\g}{{\mathfrak g}}
\renewcommand{\b}{{\mathfrak b}}

\newcommand{\pt}{\text{\rm pt}}
\DeclareMathOperator{\id}{Id}
\DeclareMathOperator{\Id}{Id}
\newcommand{\red}{\text{\rm red}}
\newcommand{\vb}{\text{\rm vb}}
\newcommand{\comp}{\text{\rm comp}}

\newcommand{\wdg}{\wedge}
\newcommand{\KK}{{\mathbb K \mathbb K}}
\newcommand{\K}{{\rm K}}

\DeclareMathOperator{\Ind}{Ind}
\DeclareMathOperator{\Res}{Res}
\DeclareMathOperator{\IndRes}{IndRes}
\DeclareMathOperator{\Index}{Index}
\DeclareMathOperator{\Th}{Th}

\begin{document}   

\title*{
Weyl  Character Formula in KK-Theory}  
 
\date{\today}

\author{Jonathan Block\inst{1} and Nigel Higson\inst{2}}  

\institute{Department of Mathematics, University of Pennsylvania,  209 South 33rd Street 
Philadelphia, PA 19104, USA.
\texttt{blockj@math.upenn.edu}
\and Department of Mathematics, Penn State University, University Park, PA 16803, USA. \texttt{higson@math.psu.edu}}

\maketitle

\section{Introduction}

The purpose  of this paper is to begin an exploration of  connections between the Baum-Connes conjecture in operator $\K$-theory and the geometric representation theory of reductive Lie groups.  Our initial goal is very modest, and we shall not stray far from the realm of compact groups, where geometric representation theory amounts to elaborations of the Weyl character formula such as the Borel-Weil-Bott theorem.  We shall  recast the topological $\K$-theory approach to the Weyl character formula, due basically to Atiyah and Bott, in the language of Kasparov's $\K\K$-theory \cite{MR582160}.  Then we shall show how, contingent on the Baum-Connes conjecture, our  $\K\K$-theoretic  Weyl character formula can be carried over to noncompact groups.

The current form of the Baum-Connes conjecture is presented in \cite{MR1292018}, and the case of Lie groups is discussed there in some detail.   On the face of it, the conjecture is   removed from traditional issues in representation theory,  since it concerns the $\K$-theory of group $C^*$-algebras, and therefore projective or quasi-projective  modules over group $C^*$-algebras, rather than for example irreducible $G$-modules.  But an important connection with the representation theory of reductive Lie groups was evident from the beginning. The conjecture uses the  \emph{reduced} group $C^*$-algebra, generated by convolution operators on $L^2 (G)$, and as a result discrete series representations are projective in the appropriate $C^*$-algebraic sense.  In fact each   discrete series contributes a distinct generator to $C^*$-algebra $\K$-theory, and the Baum-Connes conjecture is very closely related to the problem of realizing discrete series as spaces of harmonic spinors on the symmetric space associated to $G$ (the quotient of $G$ by a maximal compact subgroup $K$).  Indeed this insight contributed in an important way to the formulation of the conjecture in the first place.  For further discussion see \cite[Sec.\ 4]{MR1292018} or \cite[Sec.\ 2]{MR1957086}.

We shall not be directly concerned here with either discrete series representations or the symmetric space.  Instead the quotient space $G/K$ will appear in connection with a comparison between $K$-equivariant and $G$-equivariant $\K$-theory.   We  shall interpret the Baum-Connes conjecture as the assertion that a natural restriction map from $G$-equivariant $\K\K$-theory to $K$-equivariant $\K\K$-theory is an isomorphism.  This will allow us to carry over calculations for the compact group $K$ to the noncompact group $G$.

We shall need to adopt definitions of $\K\K$-theory that are a bit different from those of Kasparov in \cite{MR582160} and \cite{MR918241}.  We shall take as our starting point the concept, first used by Atiyah \cite{MR0228000}, of a transformation from $\K_G(X)$ to $\K_G(Y)$ that is in a suitable sense continuous at the level of cycles, so that, for example,  it extends to families of cycles and  hence determines  transformations
\[
\K_G(X\times Z)\longrightarrow \K_G (Y\times Z).
\]
The right sort of continuity can be captured fairly easily using the multiplicative structure of $\K$-theory (there are other ways to do it as well).  We define $\KK_G(X,Y)$ to be the abelian group of all   continuous transformations from $\K_G(X)$ to $\K_G(Y)$. 

If $G$ is a \emph{compact} group, then the equivariant $\KK$-theory defined in this way has an extremely straightforward geometric character, sketched out in Section~\ref{sec-kk-theory}. It  amounts to a  framework in which to organize the basic constructions of Atiyah and Hirzebruch such as those in \cite{MR0110106} that    initiated $\K$-theory.  Moreover it happens that for reasonable spaces,  such as the closed complex $G$-manifolds that we shall be considering, we recover Kasparov's theory as a consequence of Poincar\'e dualty.  We hope that our discussion of these things in Section~\ref{sec-kk-theory} will serve as a helpful introduction to $\K\K$-theory for some.

 The utility of $\KK$-theory is illustrated     by the geometric approach to the Weyl character formula for a connected and compact Lie group.  This is basically due to  Atiyah and  Bott in \cite{MR0232406}, but we shall give our own presentation of their ideas in Section~\ref{sec-weyl}.  The main novelty   is the interpretation of the numerator in the Weyl character formula in terms of intertwining operators.  We shall also make additional remarks related to the Demazure character formula \cite{MR0430001} and the determination of the Kasparov ring $\KK_G(\Flags,\Flags)$, where $\Flags$ is the flag variety of $G$.  
 
Here is a rapid summary of our Weyl character formula in equivariant $\KK$-theory (or equivalently in Kasparov's theory, since the groups involved are the same). There are intertwining operators 
\[
I_w \colon \K_G(\Flags)\longrightarrow \K_G(\Flags)
\]
associated to each element of the Weyl group of $G$.  There is a global sections operator 
\[
\Gamma \colon \K_G (\Flags)  \longrightarrow \K_G (\pt ).
\]
given by the wrong-way map in $\K$-theory determined by the collapse of $\Flags$ to a point.  According to the Atiyah-Singer  index theorem, $\Gamma$  maps the $\K$-theory class of a holomorphic $G$-vector bundle $E$ on $\Flags$ to the alternating sum of  the Dolbeault cohomology groups on $\Flags$ with coefficients in $E$.
Finally there is a localization operator
\[
\Lambda \colon  \K_G (\pt )\longrightarrow \K_G (\Flags) .
 \]
 To each point of the flag variety $\Flags$ there is associated a nilpotent subalgebra $\n$ of $\g$, and $\Lambda$ takes a representation $V$ to the alternating sum of the bundles on $X$ with Lie algebra homology fibers $H_*(\n,V)$.
 We obtain elements
 \[
\Lambda \in \KK_G( \pt,\Flags) , \quad \Gamma \in \KK_G (\Flags,\pt) , \quad \text{and}\quad I_w\in \KK_G(\Flags,\Flags) ,
\]
and we shall prove that 
\[
 \Gamma \otimes  \Lambda =\sum _{w\in W} I_w    \,\, \in\,\,   \KK_G(\Flags,\Flags) ,
\]
while
\[
\Lambda \otimes \Gamma = |W| \cdot \id       \,\,\in   \,\, \KK_G(\pt,\pt) .
\]
Here $\otimes$ is the Kasparov product---composition of transformations, but written in the reverse order.
 The connection to Weyl's formula will be clear to experts, but is in any case explained in Section~\ref{sec-comparison}.
 
Now let  $G$ be a connected linear reductive Lie group,     let $\Flags$ be the complex flag variety associated to $G$, and let $K$ be a maximal compact subgroup of $G$.   There is a restriction functor 
\begin{equation}
\label{eq-bc-restr}
\KK_G(X,Y)\longrightarrow \KK_K(X,Y)
\end{equation}
in (our version of) equivariant $\K\K$-theory, and, as we noted above,  the Baum-Connes conjecture is equivalent to the assertion that the restriction functor is an \emph{isomorphism}. So if we assume the Baum-Connes conjecture, then we can argue, by passing  from a compact form $G_\comp$ of $G$ to the subgroup $K\subseteq G_\comp$,  and then to $G$ itself by means of the restriction isomorphism (\ref{eq-bc-restr}), that the Weyl character formulas given above carry over to $G$-equivariant $\K\K$-theory.  In doing so, we shall use the full strength of the $\K\K$-theoretic Weyl character formula.  The resulting formulas could be thought of as a  $\K$-theoretic version of the character formula of Osborne \cite{MR0439996,MR0492083,MR716371}, although we shall not explore that  here.  They also call to mind the mechanisms of geometric representation theory \cite{MR610137,MR804725}, involving as they do a ``localization functor'' $\Lambda$ and a ``global sections functor'' $\Gamma$ that are essentially  inverse to one another (at the level of $W$-invariants).  

Among our longer term goals are refinements of the Weyl character formulas presented here,  perhaps
using concepts from \cite{MR2648899}, that seek to bridge between analytic tools familiar in the Baum-Connes theory and   tools familiar in geometric representation theory. This larger project we are pursuing in collaboration with David Ben-Zvi and David Nadler.

\paragraph{Acknowledgements}  We would like to thank David Ben-Zvi,  Dragan Milicic, David Nadler and David Vogan for    stimulating and helpful conversations on aspects of this paper.   The work reported on here was partially supported by the American Institute of Mathematics through its SQuaREs program, and by grants from the National Science Foundation.

\section{Basic KK-Theory}
\label{sec-kk-theory}

In this section we shall  review some basic constructions in topological $\K$-theory,   and then recast them within $\K\K$-theory.  We shall do so partly  to introduce  $\K\K$-theory from a geometric point of view that might appeal to some, and partly to set the stage for a discussion of the Baum-Connes conjecture in the last section of the paper.

\subsection{Very Basic K-Theory}

If $X$ is a compact Hausdorff space,  then the Atiyah-Hirzebruch topological $\K$-theory group $\K(X)$    is the Grothendieck group of  complex  vector bundles on $X$.   The pullback operation on vector bundles makes $\K(X)$ into a contravariant functor, and indeed a contravariant \emph{homotopy}  functor.  See for example \cite{MR0224083}.  

The $\K$-theory functor extends to the category of locally compact Hausdorff spaces.  In this category  the morphisms from $W$ to $Z$ are the continuous maps from the one-point compactification of $W$ to the one-point compactification of $Z$.  For example if $U$ is an open subset of a locally compact space $X$, then the map that collapses the complement of $U$ in $X$ to the point at infinity gives a map
\begin{equation}
\label{eq-push-forward}
\K(U)\longrightarrow\K(X).
\end{equation}
We'll call this the \emph{pushforward} from $\K(U)$ to $\K(X)$.

A compactly supported and bounded complex of vector bundles
\begin{equation}
\label{eq-cplx-over-w}
E_0\longleftarrow E_1 \longleftarrow \cdots \longleftarrow E_k
\end{equation}
over a locally compact space $W$  determines an element of $\K(W)$.  Here the \emph{support}  of a complex is the complement of the largest open set  in the base  over which each complex of fibers is exact.  In fact $\K(W)$ may be completely described in terms of such complexes; see \cite{MR0234452} or \cite{MR0236950}.

\begin{example} 
\label{ex-thom}
Let  $V$ be a   vector bundle over a locally compact base space.  Consider the following complex over the total space of $V$, in which the indicated bundles have been pulled back from the base to the total space:
\begin{equation}
\label{eq-bott-cplx}
\wdg ^0 V^* \stackrel{\iota_r}\longleftarrow \wdg ^1 V^*  \stackrel{\iota_r}\longleftarrow \cdots  \stackrel{\iota_r}\longleftarrow \wdg ^{\text{top}} V^* .
\end{equation}
The differentials at $v\in V$ are given by contraction with the \emph{radial field} on $V$.  This is the section of the pullback of $V$  whose value at the point $v$ is the vector $v$.    The complex \eqref{eq-bott-cplx} is exact except at the set of zero vectors.  So if the base  is compact, then the  complex is compactly supported and hence determines an element  of $\K(V)$, called the \emph{Thom element}. As a special case, if the original base space is a point, then we obtain   the \emph{Bott element} in $\K(V)$. 
\end{example}

Returning to the case of a general complex (\ref{eq-cplx-over-w}),   we obtain a  class in $\K(U)$ for any open subset $U$ of the base that contains the support of the complex. These $\K(U)$-classes are related to one another (as the open set $U$ varies) by   pushforward maps.

\subsection{Products and KK-Groups}
Continuing our review, the tensor product operation on vector bundles over compact spaces determines a functorial and associative product operation
\begin{equation}
\label{eq-k-theory-prod}
\K(X)\otimes \K(Y)\longrightarrow \K(X\times Y).
\end{equation}
We'll denote the product   by $E\boxtimes F$.  The   product of two $\K$-theory classes given by compactly supported complexes is given by the tensor product of complexes.

Combining (\ref{eq-k-theory-prod}) with restriction to the diagonal in $X\times X$, we find that  $\K(X)$ carries the structure of a commutative ring.  More generally, associated to a map  $Y\to X$ there is a $\K(X)$-module structure on $\K(Y)$.  We'll write this product as $E\cdot F$.  If $X$ is compact then the class $1_X\in \K(X)$ of the trivial rank-one bundle is a multiplicative identity in $\K(X)$.  

For the purposes of the following definition  fix the locally compact spaces $X$ and $Y$ and consider the groups 
\[
\K(X\times Z)\quad \text{and} \quad \K(Y\times Z)
\]
as contravariant functors in the $Z$-variable.  

\begin{definition}
\label{def-cont-transf}
A natural transformation 
\begin{equation}
\label{eqn-cont-tr1}
T_Z \colon \K(X\times Z)\longrightarrow \K(Y\times Z)   
\end{equation}
is \emph{continuous} if each $T_Z$ is a $\K(Z)$-module homomorphism.
\end{definition}

 The term ``continuous'' has been chosen to suggest that $T$ is a means to construct (virtual) vector bundles on $Y$ from vector bundles on $X$ that is continuous in the sense that it extends to families (think of a vector bundle on $X\times Z$ as a family of vector bundles on $X$ parametrized by $Z$).

\begin{lemma}
\label{rem-equiv-cont-lemma}
A natural transformation  $T$ is   {continuous} if and only if all diagrams of the form
\begin{equation}
\label{eqn-cont-tr2}
\begin{matrix}
 K(X\times Z)\otimes K(W)  &\xrightarrow{T_Z\otimes \id}  &  K (Y\times Z)  \otimes K(W) \\
 \downarrow && \downarrow \\
  K(X\times Z\times W) & \xrightarrow[T_{Z\times W}]{}& K (Y\times Z\times W)
  \end{matrix}
\end{equation}
are commutative.
\end{lemma}

\begin{proof}
To recover Definition~\ref{def-cont-transf} from  (\ref{eqn-cont-tr2}), set $W=Z$  and then restrict to the diagonal in $Z\times Z$.  In the reverse direction, it suffices to prove  (\ref{eqn-cont-tr2})  from Definition~\ref{def-cont-transf} when  $W$ is compact.  Consider the commuting diagram 
\begin{equation*}
\begin{matrix}
 K(X\times Z)   &\xrightarrow{T_Z}  &  K (Y\times Z)  \otimes K(W) \\
 \downarrow && \downarrow \\
  K(X\times Z\times W) & \xrightarrow[T_{Z\times W}]{}& K (Y\times Z\times W)
  \end{matrix}
\end{equation*}
in which the vertical maps are induced from the projection $Z\times W\to Z$.  The commutativity of (\ref{eqn-cont-tr2}) follows from the fact that the bottom horizontal map is a $K(Z\times W)$-homomorphism. \qed
\end{proof}

\begin{definition}
Denote by $\KK(X,Y)$ the abelian group of continuous transformations from $\K(X)$ to $\K(Y)$. 
\end{definition}

\begin{remark}
This is  \emph{not} Kasparov's definition of $\K\K$-theory \cite{MR582160}.  However Kasparov's $\K\K$-groups are closely related to our $\KK$-groups, and indeed they are the same in many cases; see Section~\ref{sec-kasp-theory}.
\end{remark}

It is evident that there is an associative   ``product'' operation
\begin{equation}
\label{eq-kasp-prod}
\KK (X,Y)\otimes \KK(Y,Z)\longrightarrow \KK (X,Z) 
\end{equation}
given by composition of transformations.  
This is the \emph{Kasparov product}, and using it we obtain an additive category whose objects are locally compact spaces and whose morphisms are continuous transformations. 
However the  Kasparov product of elements 
\[
S\in \KK(X,Y)\quad \text{and}\quad T \in \KK(Y,Z)
\]
 is conventionally denoted  
 \begin{equation}
 \label{eq-kasp-prod2}
 S \otimes T\in \KK (X,Z),
 \end{equation}
  which is opposite to the notation  $T\circ S$ for composition of morphisms in a category.    
 
In addition to the Kasparov product (\ref{eq-kasp-prod}), there is by definition a pairing 
\begin{equation}
\label{eq-kasp-prod-k-thy}
\K(X) \otimes \KK(X,Y)\longrightarrow \K(Y) 
\end{equation}
given by the action of continuous transformations on $\K$-theory.  We'll use the same symbol $\otimes$ for it, and indeed  it can be viewed as a special case of the Kasparov product since the evaluation of the pairing 
\[
\K(\pt) \otimes \KK(\pt,Y)\longrightarrow \K(Y) 
\]
on the unit element $1_\pt\in \K(\pt)$ gives an  isomorphism of abelian groups
\[
\KK(\pt, Y)\stackrel \cong \longrightarrow \K(Y) .
\]

Some   examples of continuous transformations, and hence $\KK$-classes, immediately suggest themselves.

\begin{example}
 If $f\colon Y\to X$ is a continuous and proper map, then the pullback construction for vector bundle determines a class 
\begin{equation}
\label{eq-pullback-class}
[X\stackrel{f}{\leftarrow} Y] \in \KK (X,Y) .
\end{equation}
If $g\colon Z\to Y$ is another continuous and proper map, then
\[
[X\stackrel{f}{\leftarrow} Y] \otimes [Y\stackrel{g}{\leftarrow} Z]  = [X\stackrel{f\circ g}{\leftarrow} Z]  .
\]
\end{example}

\begin{example}
A complex vector bundle $E$  on a compact space $X$, or more generally a class in $\K(X)$, where $X$ is  any locally compact space, determines a continuous  transformation  
\begin{equation}
\label{eq-vb-class}
[E]\in \KK(X,X)
\end{equation}
by multiplication using the product $\cdot$  in $\K$-theory.\footnote{We shall reserve the square bracket notation $[\,\,\,]$ for  $\KK$-classes, and avoid it for  $\K$-theory classes.}   We obtain in this way an injective  
ring homomorphism
\begin{equation}
\label{eq-vb-class2}
 \K( X)\longrightarrow  \KK(X,X) ,
\end{equation}
where the multiplication in the ring  $\KK(X,X)$ is of course the Kasparov product.\end{example}

\begin{example}
If $V$ is a complex vector bundle over a locally compact space $W$, then the \emph{Thom homomorphism}
\[
\Th\colon \K(W)\longrightarrow \K(V)
\]
is given by pullback to $V$ (of vector bundles or complexes of vector bundles initially defined on $W$) followed by product with the Thom complex (\ref{eq-bott-cplx}). Note that even when $W$ is merely locally compact, the result is a compactly supported complex, because the support of a tensor product of complexes is the intersection of the supports of the factors.
If $Z$ is an auxiliary space, then by pulling back $V$ to $W\times Z$ we obtain in addition a Thom homomorphism 
\begin{equation}
\label{eq-thom-kk-class}
\Th_Z\colon \K(W\times Z)\longrightarrow \K(V\times Z) .
\end{equation}
Collectively these define a continuous transformation $\Th\in \KK(W,V)$. 
\end{example}

\begin{remark}
The concept of continuous transformation (although not the name)  has its origins  in  Atiyah's elliptic operator proof of Bott periodicity \cite{MR0228000}.  Let $V$ be a finite-dimensional complex vector space.  According to Atiyah's famous rotation trick,   if two continuous transformations 
\[
\K(\pt)\longrightarrow \K(V) \quad \text{and} \quad \K(V)\longrightarrow \K(\pt)
\]
compose to give the identity on $\K(\pt)$, then they compose the other way to give the identity on $\K(V)$ (and thereby implement Bott periodicity).
\end{remark}

 For later purposes we note the following fact: 
 
\begin{lemma} 
\label{lemma-thom-funct}
Every continuous transformation  is functorial with respect to Thom homomorphisms.
\end{lemma}

\begin{proof}
We mean that if $V$ is a complex vector bundle over a locally compact space $W$, and if $T\in \KK(X,Y)$, then the diagram
\[
\xymatrix{
 \K(X\times W) \ar[r]^{T_W} \ar[d]_{\Th_X}&\K(Y\times W)\ar[d]^{\Th_Y} \\
 \K(X\times V)\ar[r]_{T_V} & \K(Y\times V)
 }
\]
is commutative.  If $W$ is compact, so that the Thom class is defined in $\K(V)$,  then this is a  consequence of the definition of continuous transformation (in the form given in Lemma~\ref{rem-equiv-cont-lemma}). The noncompact case can   be reduced to the compact case. 
\qed
\end{proof}
 
\subsection{Wrong-Way Maps}

Let $X$ and $Y$ be complex or almost-complex manifolds (without boundary) and let $f\colon X \to Y$ be any continuous map.  There is an associated \emph{wrong-way class}
 \begin{equation}
 \label{eq-wrong-way-class}
[X\stackrel{f}\to Y]\in \KK(X,Y) ,
\end{equation}
which determines a \emph{wrong-way map} from 
$ \K(X)$ to $ \K(Y)$.  The construction, which is due to Atiyah and Hirzebruch, is based on the Thom homomorphism and Bott periodicity.

Recall that the normal bundle of an embedding of manifolds $h\colon X \to Z$ is the quotient real vector bundle 
\begin{equation}
\label{eqn-normal-bdle}
N_ZX =    h^* TZ \,\, / \,  TX,
\end{equation}
where the derivative of $h$ is used to embed $TX$ into $h^* TZ$.
If $X$ is almost-complex, then we may embed $X$ equivariantly into a finite-dimensional complex   $G$-vector space $V$  in such a way that the normal bundle  $N_VX$ admits a complex vector bundle structure for which there is an isomorphism 
\begin{equation}
\label{eqn-normal-bdle2}
N_VX\oplus TX  \cong  V{\times}X
\end{equation}
of  complex vector bundles. The normal bundle for the diagonal map of $X$ into $ Y\times V$ is as follows:
\begin{equation}
\label{eq-normal-bdle3}
N_{Y\times V}X \cong  f^*TY\oplus N_VX 
\end{equation}
 (from the definition (\ref{eqn-normal-bdle}) there is a natural projection map from the left-hand side onto $N_VX$, and a natural isomorphism from $f^*TY$ onto the kernel of this projection).  In particular it too is endowed with a complex structure.  The wrong way map from $K(X)$ to $K(Y)$ is by definition  the composition
 \begin{equation}
 \label{eq-wrong-way2}
 \K(X)\longrightarrow \K(N_{Y\times V} X)\longrightarrow \K(Y\times V) \longrightarrow \K(Y) ,
 \end{equation}
 in which the first map is the Thom homomorphism, the second is induced from an embedding of the normal bundle as a tubular neighborhood, and the last is the Bott periodicity isomorphism.  
 Each map in (\ref{eq-wrong-way2}) is given by some $\KK$-class, and the Kasparov product of these   is the wrong-way class \eqref{eq-wrong-way-class}.
 
 The wrong-way class   depends only on the homotopy class of $f$ and is functorial with respect to ordinary composition of smooth maps:
 \[
 [X\stackrel f \to Y]\otimes [Y\stackrel g \to Z] = [X\stackrel {g\circ f}\longrightarrow Z]
 \in \KK(X,Z).
 \]

 \begin{example}
 \label{ex-integrality}
Wrong-way maps in  $\K$-theory (and for that matter $\K$-theory itself) were introduced by Atiyah and Hirzebruch in \cite{MR0110106}  to prove integrality theorems for characteristic numbers.\footnote{Actually the construction of wrong-way maps was only conjectural at the time that \cite{MR0110106} was written, but see  Karoubi's book \cite[Chapter IV, Section 5]{MR0488029} for a full exposition.}    It may be calculated that if $X$ is a closed almost-complex manifold, and if $E$ is a complex vector bundle on $X$, then the wrong-way map associated to $X\to \pt$ sends the $\K$-theory class of $E$ to 
 \begin{equation}
 \label{eqn-rr}
 \pi_*(E ) = \int_X \operatorname{Todd}(TM) \operatorname{ch}(E) \cdot 1_\pt\in K(\pt).
  \end{equation}
As a result the integral is always an integer.  
 \end{example}
 
 \begin{example}
 If $ U\to X$ is the embedding of an open subset into $Y$, then the wrong-way class $[U\to X]$ is the same as the   class $[U\leftarrow X]$ associated to the pushforward construction of Example~\ref{eq-push-forward}.
 \end{example}
  
\begin{example}
 The wrong-way map for the inclusion of the zero section in a complex vector bundle is the Thom homomorphism.  The fact that it  is an isomorphism reflects the fact that   the   section map is a homotopy equivalence.
\end{example}

The wrong-way classes are related to the pullback classes of \eqref{eq-pullback-class} in a number of ways, of which the most important is this:

 \begin{lemma}
 \label{lemma-base-change}
If in  the diagram 
 \[
\xymatrix@C=12pt@R=12pt{  Z\times_Y W \ar[d] \ar[r] & W\ar[d] \\
Z\ar[r] & Y
}
 \]
of almost-complex manifolds the map from $Z$ to $Y$ is a submersion, 
then \begin{equation}
\label{eq-base-change}
 [Z \rightarrow Y]\otimes [Y\leftarrow  W] = [Z \leftarrow Z\times_Y W] \otimes [Z\times_Y W \rightarrow W]
 \end{equation}
in $\KK(Z,W)$. 
 \end{lemma}

The lemma follows from the definitions, together with the fact that the normal bundle associated, as in (\ref{eq-normal-bdle3}), to the top horizontal map is the pullback of the normal bundle associated to the bottom horizontal map.

\subsection{Poincar\'e Duality}
\label{sec-poincare}

We shall calculate $\KK(X,Y)$ when $X$ and $Y$ are closed (almost) complex manifolds.

\begin{lemma} 
\label{lemma-cont-wrong-way}
Continuous transformations are functorial  for   wrong-way maps in $\K$-theory.
\end{lemma}

\begin{proof}
The assertion in the lemma is that if  $T\in \KK(X,Y)$, and if $f\colon W\to Z$ is a  continuous map between almost-complex manifolds, then the diagram
\[
\xymatrix{
\K(X\times W) \ar[r]^{T_W}\ar[d]_{f_*}&  \K(Y\times W) \ar[d]^{f_*}\\
\K(X\times Z) \ar[r]_{T_Z} & \K(Y\times Z)
}
\]
in which the vertical arrows are wrong-way maps, is commutative.
 This follows from   functoriality with respect to  the Thom homomorphism, as described in Lemma~\ref{lemma-thom-funct}. \qed
\end{proof}

Denote by $\Delta_X \in \K(X\times X)$ the image of the unit class $1_X\in \K(X)$ under the wrong-way map induced from the diagonal map  $\delta\colon X \to X\times X$: 
 \begin{equation}
 \label{eq-diag-class}
 \Delta_X : = 1_X  \otimes [X\stackrel{\delta}{\to} X\times X] \in \K (X\times X).
 \end{equation}

\begin{definition}
Let $Y$ be any locally compact  space.  The \emph{Poincar\'e duality homomorphism}
 \begin{equation}
 \label{eq-pd-map}
 \widehat{\phantom{-}}\colon \KK(X,Y)\longrightarrow \K(Y\times X)
\end{equation}
is defined by means of the formula 
\[
 \widehat T = T_X(\Delta_X) \in \K(Y\times X) ,
\]
 where  $T\in \KK(X,Y)$.
 \end{definition}
 
 \begin{example} The Poincar\'e  dual of a wrong-way class $[X\stackrel f \to Y]\in KK(X,Y)$ is the class
 \[
1_X\otimes [X\to Y\times X] \in K(Y\times X)
\]
obtained from the embedding of $X$ into $Y\times X$ as the graph of $f$.
 \end{example}
 
 \begin{proposition}
 \label{prop-poincare}
 The Poincar\'e duality homomorphism  \textup{(\ref{eq-pd-map})} is an isomorphism of abelian groups.
 \end{proposition}
 
 \begin{proof}
 A duality map  in the reverse direction,
 \begin{equation}
 \label{eq-reverse-duality}
 K(Y\times X)\longrightarrow \KK(X,Y) ,
 \end{equation}
 may be defined by sending a vector bundle (or $\K$-theory class) $E$ on $Y\times X$ to the Kasparov product
 \[
  [X\leftarrow Y\times X] \otimes [E]\otimes [X\times Y\to Y] \in \KK(X,Y) .
  \]
  Applying first (\ref{eq-pd-map}) and then (\ref{eq-reverse-duality}) we obtain from a continuous transformation $T\in \KK(X,Y)$ the continuous transformation
 \begin{equation}
  \label{eq-pd-comp1}
  K(X )\ni a\mapsto  p_*\bigl( (1_Y\boxtimes a)\cdot  T_{X}(\Delta_X)  \bigr) \in K(Y )  ,
  \end{equation}
   where $p_*\colon \K(Y\times X)\to \K(Y)$ is the wrong way map associated to the projection
\[
p\colon Y\times X \longrightarrow Y 
\]
 onto the first factor  (to keep the notation simple we are suppressing the auxiliary space $Z$). 
  By   definition of continuous transformation, the right-hand side of (\ref{eq-pd-comp1}) is equal to 
  \begin{equation}
  \label{eq-pd-comp2}
   p_*\bigl(   T_{X} ((1_X\boxtimes a)\cdot \Delta_X) \bigr)  .
  \end{equation}
  Now 
  \begin{equation}
  \label{eq-pd-comp3}
 (a\boxtimes 1_X)\cdot \Delta_X =   (1_X\boxtimes a)\cdot \Delta_X 
  \end{equation}
  because the class  $\Delta_X\in K(X\times X)$ is supported on the diagonal, and because the two coordinate projections from $X\times X$ to $X$ are homotopic in a neighborhood of the diagonal.  Using (\ref{eq-pd-comp3}) together with Lemma~\ref{lemma-cont-wrong-way}, we find that   (\ref{eq-pd-comp2}) is equal to 
$
  T(q_*(\Delta_X)\cdot a)
$,
 where $q_*$ is the wrong-way map associated to the projection 
 \[
 q\colon X\times X\to X
 \]
onto the first factor.  Since $q_*(\Delta_X)=1_X$, we recover from (\ref{eq-pd-comp1}) the continuous transformation $T$ with which we began, as required.
  
 The calculation of the composition of the two duality maps in the other order is similar. \qed

 \end{proof} 
 
 \begin{remark}
 A similar discussion is possible for manifolds with boundary.
 Suppose that $\overline X$ is a compact almost-complex manifold with boundary, and denote by $X$ its interior.  Define a modified diagonal map 
 \[
 \begin{gathered}
 \overline X\longrightarrow \overline X \times X \\
 x \mapsto (x, f(x))
 \end{gathered}
 \]
 using any $f\colon \overline{X} \to X$ for which the compositions
 \[
 \overline{X}\longrightarrow X \longrightarrow \overline{X}\quad \text{and}
 \quad
  X \longrightarrow \overline{X} \longrightarrow X
  \]
 with the inclusion of $X$ into $\overline{X}$ are homotopic to the identity (for example, define $f$ by deforming a closed collar of the boundary into its interior).  Define $\Delta_{X}\in \K(\overline{X}\times X)$ to the image of the class $1_{\overline{X}}\in \K(\overline{X})$ under the associated wrong way map, and then define the \emph{Poincar\'e-Lefschetz duality map}
 \begin{equation}
 \label{eqn-pd-map2}
 \widehat{\phantom{-}}\colon \KK(\overline X,Y)\longrightarrow \K(Y\times X)
\end{equation}
by the same formula as before.   It too is an isomorphism, by the same argument as before.
\end{remark}

\subsection{Equivariance with Respect to a Compact Group}
\label{sec-equivariance}

If $G$ is a compact Hausdorff group, and if $X$ is a compact or locally compact Hausdorff $G$-space, then we can form the equivariant topological $\K$-theory group $\K_G(X)$; see  \cite{MR0234452}.  We can repeat our discussion up to now using $\K_G(X)$ in place of $\K(X)$ (and, for example, equivariant continuous maps between complex $G$-manifolds, and so on), including the definition of equivariant  groups $\KK_G(X,Y)$ comprised of continuous transformations from the equivariant $\K$-theory of $X$ to the equivariant $\K$-theory of $Y$.  Our discussion carries through without change.

New issues arise related to restriction and induction.  If  $H$ is a closed subgroup of $G$, then there is a  \emph{restriction map}
\begin{equation}
\label{eq-restr-map}
\K_G(X)\longrightarrow \K_H(X)
\end{equation}
that forgets $G$-equivariance and retains only $H$-equivariance.  
In the other direction there is an \emph{induction map}
\begin{equation}
\label{eq-ind-map}
\K_H(X)\longrightarrow \K_G(\Ind_H^G X ).\end{equation}
Here $
\Ind_H^G X$ is  the quotient of $G\times X$ by the $H$-action $h\colon (g,x)\mapsto (gh^{-1},hx)$, and the induction map is obtained by associating to an $H$-equivariant bundle $E$ over $X$ the $G$-equivariant bundle $\Ind_H^G E $ over $\Ind _H^G X$. 
The induction map (\ref{eq-ind-map})  is an \emph{isomorphism}: its inverse is given by restriction, followed by pulling back along the $H$-equivariant map from $X$ into $\Ind_H^G X$ that sends $x\in X$ to $(e,x)\in\Ind_H^G X$.

Using the induction isomorphism  (\ref{eq-ind-map}) we   define a restriction homomorphism 
\begin{equation}
\label{eq-kk-resr-map}
\Res_H^G\colon \KK_G(X,Y)\longrightarrow \KK_H(X,Y)
\end{equation}
as follows.  If $X$ is a $G$-space and $Z$ is an $H$-space, then there is a $G$-equivariant homeomorphism
\begin{equation}
\label{eq-ind-homeo}
\Ind_H^G  (X\times Z) \stackrel  \cong\longrightarrow    X\times \Ind_H^G Z
\end{equation}
induced from $(g,x,z) \mapsto (gx, g,z)$.  If $T\in \KK_G (X,Y)$, then define a continuous transformation in $H$-equivariant $\K$-theory by means of the diagram
\begin{equation}
\label{eq-ind-homo2}
\xymatrix{
\K_H (X\times Z) \ar[r]^{(\Res_H^G T)_Z}\ar[d]_\cong&  \K_H(Y\times Z)\ar[d]^\cong \\
\K_G(\Ind_H^G (X\times Z))  & \K_G (\Ind_H^G(Y\times Z)) \\
\K_G (  X\times \Ind_H^G Z) \ar[u]^\cong\ar[r]_{T_{\Ind_H^G Z}} & \K_G (  Y\times \Ind_H^G Z)  \ar[u]_\cong.
}
\end{equation}
The pullback, vector bundle and wrong-way $\KK_G$-classes defined in (\ref{eq-pullback-class}), (\ref{eq-vb-class}) and (\ref{eq-wrong-way-class})  are mapped to like $\KK_H$-classes obtained by forgetting $G$-equivariant structure and retaining only $H$-equivariant structure.  Moreover restriction from $\KK_G$ to $\KK_H$ is compatible with Kasparov product.

\begin{remark}
\label{rem-restr-in-kk}
Note for later use that the definition of the restriction homomorphism in $\KK$ does not involve the restriction homomorphism in $\K$-theory.
\end{remark}

There is also an induction homomorphism 
\begin{equation}
\label{eq-ind-map-kk}
\Ind_H^G\colon \KK_H(X,Y)\longrightarrow \KK_G(\Ind^G_H X, \Ind^G_H Y)
\end{equation}
that is defined in a very similar way,  by means of the commuting diagram
\[
\xymatrix{
\K_H (X\times Z) \ar[r]^{T_Z}\ar[d]_\cong&  \K_H(Y\times Z)\ar[d]^\cong \\
\K_G(\Ind_H^G (X\times Z))  & \K_G (\Ind_H^G(Y\times Z)) \ar[d]^\cong \\
\K_G (  \Ind_H^G X\times  Z) \ar[u]^\cong\ar[r]_{(\Ind_H^G T)_{Z}} & \K_G (   \Ind_H^GY\times  Z)  \ar[u]_\cong.
}
\]
Here $Z$ is a $G$-space, viewed in the top row as an $H$-space by restriction.  The top vertical maps are the induction isomorphisms (\ref{eq-ind-map}) and the bottom result from the identifications
\[
\Ind_H^G  (X\times Z) \stackrel  \cong\longrightarrow     \Ind_H^G X\times  Z
\]
given by $(g,x,z) \mapsto (g, x,gz)$.  The induction map is compatible with the construction of pullback and vector bundle classes in (\ref{eq-pullback-class}) and (\ref{eq-vb-class}).  If $G/H$ is given a $G$-invariant almost-complex structure (should one exist), then $\Ind_H^G X$ will acquire a $G$-invariant almost-complex structure from an $H$-invariant almost complex structure on $X$, and the induction map becomes compatible with the wrong-way classes (\ref{eq-wrong-way-class}) too.

Neither (\ref{eq-kk-resr-map}) nor (\ref{eq-ind-map-kk}) is an isomorphism.  But if $X$ is a $G$-space and   $Y$   is an $H$-space, then the hybrid map 
\begin{equation}
\label{eq-kk-indres-map}
\IndRes \colon  \KK_G(X,\Ind_H^G Y)  \longrightarrow  \KK_H(X,Y)  
\end{equation}
 that is obtained by first restricting from $G$ to $H$, as in (\ref{eq-kk-resr-map}),  and then restricting to the subspace $Y\subseteq \Ind_H^G  Y$, \emph{is} an isomorphism. Its inverse is   defined by means of the diagram
\[
\xymatrix@C=65pt{
\K_G(X\times Z)\ar[r]^{(\IndRes^{-1}T)_Z}\ar[dd]_{\text{restriction}}&  \K_G(\Ind_H^G Y \times Z)\ar[d]^\cong\\
 & \K_G(\Ind_H^G(Y\times Z)) \\
 \K_H(X\times Z)   \ar[r]_{T_Z} & \K_H(Y\times Z)\ar[u]^\cong_{\text{induction}} .
 }
\]
In particular, we find that if the action of $H$ on  $Y$ extends to $G$, then restriction to $H$ followed by restriction to $eH\subseteq G/H$ gives an isomorphism
\begin{equation}
\label{eq-simple-kk-ind-iso}
\KK_G(X,G/H\times Y) \stackrel \cong \longrightarrow \KK_H(X,Y) .
\end{equation}

\subsection{The Index Theorem}
\label{sec-index-thm}
Let $X$  be a  complex (or almost-complex) $G$-manifold.  The index map for the  Dolbeault operator on $X$  is a homomorphism from $\K_G(X)$   to $\K_G(\pt)$, and it extends by a families index construction to a continuous transformation
\[
\Index_Z \colon \K_G(X\times Z)\longrightarrow \K_G(Z) .
\]
The Atiyah-Singer index theorem asserts that 
\[
\Index = [X\to \pt] \in \KK_G(X,\pt).
\]
For example, the characteristic class formula of Example~\ref{ex-integrality} describes the index of the Dolbeault operator on a closed  complex manifold $X$ (without group action) with coefficients in a smooth complex vector bundle $E$.

 More generally the Atiyah-Singer index theorem for families asserts that if $p\colon Z\to Y$ is a submersion of closed, complex manifolds, then the class
 \[
 [Z\stackrel p \to Y] \in \KK_G(Z,Y)
 \]
 is equal to the index class associated to the family of Dolbeault operators on the fibers of $p$.
 
\subsection{Kasparov's Analytic KK-Theory}
\label{sec-kasp-theory}

Kasparov's $\K\K$-groups \cite{MR582160} have their origins in the Atiyah-J\"anich theorem describing $\K$-theory in terms of families of Fredholm operators \cite{MR0248803,MR0190946}, and in Atiyah's attempt to define $\K$-homology in terms of ``abstract elliptic operators'' \cite{MR0266247}; see \cite{MR1077390} for an introduction.  By design, there is a natural transformation
\begin{equation}
\label{eq-kk-to-kk}
\K\K_G(X,Y)\longrightarrow \KK_G(X,Y)
\end{equation}
that is compatible with Kasparov products.  The classes in (\ref{eq-pullback-class}), (\ref{eq-vb-class}) and (\ref{eq-wrong-way-class}) all have   counterparts in Kasparov theory, from which it follows that $\K\K$-theory satisfies Poincar\'e duality in the way that we described in Section~\ref{sec-poincare} (in fact see \cite[Sec.\ 4]{MR918241} for much more general formulations of Poincar\'e  duality).  As a result, the map (\ref{eq-kk-to-kk}) is an isomorphism whenever $X$ is a compact almost-complex $G$-manifold (possibly with boundary). 
 A bit more generally, if $X$ is a $G$-ENR (see for example \cite{MR551743}) then $X$ is a $G$-retract of a compact almost-complex $G$-manifold with boundary, namely a suitable compact neighborhood in Euclidean space, and we see again that (\ref{eq-kk-to-kk}) is an isomorphism.

Kasparov's $\K\K$-theory gives a means   to conveniently  formulate and prove the Atiyah-Singer index theorems reviewed in the previous section.  But the real strength of $\K\K$-theory only becomes apparent when dealing with equivariant $\K$-theory for \emph{noncompact} groups $G$ \cite{MR918241}, where it provides a set of tools that operate beyond the reach of the rather simple ideas we have been developing up to now.   We shall take up this point (in a somewhat idiosynchratic way) in the last section of this paper.

\section{Weyl Character Formula}
\label{sec-weyl}

Let $G$ be a compact connected Lie group.   Our aim is to present a $\K$-theoretic account of the Weyl character formula for $G$, more or less along the lines of \cite{MR0232406}, but in a slightly more general context.

 \subsection{A First Look at KK-Theory for the Flag Variety}
\label{sec-kk-for-flags}

The Weyl character formula calculates the restrictions to a maximal torus $T\subseteq G$ of the characters of irreducible representations of $G$.  So in $\K$-theory terms it concerns the map
\begin{equation}
\label{eq-restr-map2}
\K_G(\pt)\longrightarrow \K_G(G/T)
\end{equation}
given by pulling back vector bundles over a point (that is, representations) to vector bundles over the homogeneous space $G/T$.  

One reason for recasting the character formula in these terms is that the manifold $G/T$ carries a $G$-invariant complex structure, using which the machinery of the previous sections can be applied.  For the time being let us frame this as follows: there is a complex $G$-manifold $\Flags$ and a $G$-equivariant diffeomorphism between $\Flags$ and  $G/T$.  We shall recall the construction of  $\Flags$ (the flag variety) in the next section.

Consider then the $\K$-theory classes
\[
[\pt \leftarrow \Flags] \in \KK_G (\pt, \Flags)
\]
and 
\[
[ \Flags\to \pt]\in \KK_G (\Flags, \pt) .
\]
Kasparov product with the first gives the restriction map (\ref{eq-restr-map2}).  It turns out that Kasparov product with the second gives a one sided inverse, and indeed
\begin{equation}
\label{eq-atiyah-embedding}
[\pt \leftarrow \Flags] \otimes [\Flags \to \pt] = \id \in \KK_G(\pt,\pt).
\end{equation}
This was first observed by Atiyah \cite[Prop.\ 4.9]{MR0228000} and we shall check it later (see Remark~\ref{remark-atiyah-embedding}).   So  (\ref{eq-restr-map2}) embeds $\K_G(\pt)$ as a summand of $\K_G(B)$.  Equivalently, if $T$ is a maximal torus in $G$, then the restriction map 
\[
\K_G(\pt)\to \K_T(\pt)
\]
between representation rings is an embedding onto a direct summand.

Define
 \begin{equation}
 \label{eq-def-of-pi}
\Pi =  [\Flags \to \pt]\otimes [\pt  \to \Flags] \in \KK_G(\Flags,\Flags) .
 \end{equation}
It follows from (\ref{eq-atiyah-embedding}) that $\Pi$ is an idempotent in the Kasparov ring $\KK_G(\Flags,\Flags)$, and the problem of determining the representation ring $\K_G(\pt)$ becomes (to a certain extent, at least) the problem of usefully describing the idempotent $\Pi\in \KK_G(\Flags,\Flags)$.  We shall give one solution in the next section.

\subsection{Intertwining Operators}
\label{sec-intertwiners}
Because it is a homogeneous space,  every $G$-equivariant map from $\Flags$ to itself is a diffeomorphism.  The $G$-equivariant self-maps of $\Flags$ therefore assemble into a group $W$, but we'll do so using the opposite of the usual composition law for maps, so as to obtain a right-action of $W$ on $\Flags$:
\begin{equation}
\label{eq-weyl-action}
F\times W\ni (b,w)\mapsto b\cdot w\in \Flags.
\end{equation}
This is the \emph{Weyl group} of $G$. In other words, the Weyl group is the opposite of the usual group of $G$-equivariant self-maps.  It is a finite group.    

If a basepoint $\pt\in \Flags$ is chosen, and if $T$ is the maximal torus that fixes the basepoint, then 
under the resulting identification $\Flags\cong G/T$  the action (\ref{eq-weyl-action}) takes the form
\[
G/T\times N(T)/T\ni (gT, wT) \mapsto g  wT\in G/T ,
\]
where $N(T)$ is the normalizer of $T$ in $G$.  This explains our preference for right actions in the definition of the Weyl group.  Right actions also match nicely with the Kasparov product: each Weyl group element determines a wrong-way class
\begin{equation}
\label{eq-intertwiner-def}
I_w=[\Flags \stackrel w \to \Flags]\in \KK_G(\Flags, \Flags),
\end{equation}
and we obtain a homomorphism from $W$ into the group of invertible elements in the ring $\KK_G(\Flags, \Flags)$. We'll call the elements $I_w\in \KK_G(\Flags, \Flags)$   \emph{intertwining operators} and we shall analyze them further in Section~\ref{sec-bott-sam}.

\subsection{Weyl Character Formula in KK-Theory}
\label{subsec-weyl-kk}

The ring $\KK_G(\Flags,\Flags)$ contains the intertwiners $I_w$ from (\ref{eq-intertwiner-def}) as well as the classes (\ref{eq-vb-class}) associated to $G$-equivariant vector bundles on $\Flags$.  Our aim   is to calculate the element $\Pi$ from (\ref{eq-def-of-pi}) in terms of these ingredients.

We shall use the vector bundle class
\[
[\wdg^* T^*\Flags]\in \KK_G(\Flags, \Flags).
\]
Here $T^*\Flags$ is the complex  dual of the tangent bundle $T\Flags$, and  we take degree into account, so that 
\[
[\wdg^* T^*\Flags]: = \sum (-1)^p [\wdg^p T^*\Flags]\in \KK_G(\Flags, \Flags).
\]
The representation-theoretic significance of this class  will be reviewed in the next section.  Its geometric significance is as follows.
Let $X$ be any closed complex or almost-complex manifold equipped with a smooth action of a compact torus $T$.  The $T$-fixed point set $F\subseteq X$ is a submanifold and it   inherits an almost-complex structure from $X$ (the tangent bundle $TF$ is the $T$-fixed set of  $TX\vert_F$, and this is obviously a complex subbundle of  $TX\vert_F$).

\begin{lemma} 
\label{lemma-localization2}
 If $ F\to X$ is the inclusion of the fixed-point manifold, then 
\begin{equation}
\label{eq-F-localization0}
 \wdg^*T^*X  =   \wdg^* T^*F  \otimes [F\to X] \in \K_T(X),
\end{equation}
and
\begin{equation}
\label{eq-F-localization}
[\wdg^*T^*X] = [X\leftarrow F]\otimes [\wdg^* T^*F]\otimes [F \to X] \in \KK_T(X,X ).
\end{equation}
\end{lemma}

 \begin{proof}
 Fix a $T$-invariant Hermitian metric on $X$.  Pick a vector $v$ in the Lie algebra of $T$ such that $\exp(v)$ topologically generates the torus $T$, and denote  by $v^X$ the associated Killing vector field on $X$.  Its zero set is precisely $F$.  Let $N$ be the normal bundle of $F$.  The torus $T$ also acts on $N$, and     the element $v\in \operatorname{Lie}(T)$ determines a vertical vector field $v^N$ on $N$.   The zero set of $v^N$ is precisely the set $F\subseteq N$ of zero vectors in $N$.
  
The element $\wdg^* T^*X\in \K_T(X)$ is the class of the complex of vector bundles 
  \begin{equation}
 \label{delta-eqn1}
 \wdg^0 T^*X \xleftarrow{\,\,\iota_{v^X}\,\,} \wdg ^1 T^*X \xleftarrow{\,\,\iota_{v^X}\,\,}    \cdots 
 \end{equation}
 on $X$ (since $X$ is compact the  differentials are actually irrelevant to  the $\K$-theory class).  The  support of the complex is $F$, and we   find that  $\wdg^*T^*X$ is equal to the pushforward (via a tubular neighborhood embedding) of the class in $\K_T(N)$ associated to the complex
   \begin{equation}
 \label{delta-eqn2}
 \wdg^0 T^*X \xleftarrow{\,\,\iota_{v^N}\,\,} \wdg ^1 T^*X \xleftarrow{\,\,\iota_{v^N}\,\,}   \cdots 
 \end{equation}
 Here the bundles $\wdg^p T^*X$, originally defined on $X$, are restricted to $F$ then pulled back to $N$.  On the other hand, the class $ \wdg^* T^*F \otimes [F{\to} X] $ is the pushforward from $N$ to $X$   of the class  associated to the  complex 
   \begin{equation}
 \label{delta-eqn3}
 \wdg^0 T^*X \xleftarrow{\,\,\iota_{r}\,\,} \wdg ^1 T^*X \xleftarrow{\,\,\iota_{r}\,\,}     \cdots 
 \end{equation}
 where $r$ is the radial vector field from Example~\ref{ex-thom} and (\ref{eq-bott-cplx}).
But the vertical tangent vector fields $v^N$ and $r$ are pointwise linearly independent, in fact orthogonal in the underlying euclidean metric, everywhere away from the zero vectors in $N$. So the complexes (\ref{delta-eqn2}) and (\ref{delta-eqn3}) are homotopic (in the topological sense) through complexes with support $F\subseteq N$: the differentials in the intermediate complexes   are contractions against $t v^N + (1-t)r$.  So the two complexes determine the same $\K$-theory class.   Equation \eqref{eq-F-localization} is a consequence of \eqref{eq-F-localization0}. \qed \end{proof}

In the case of the flag variety, where $F$ is zero-dimensional,  Lemma~\ref{lemma-localization2} implies that 
\begin{equation}
\label{eq-F-localization2}
[\wdg^*T^*\Flags] = [\Flags\leftarrow F] \otimes [F \to \Flags] \in \KK_T(\Flags,\Flags ) .
\end{equation}
 This leads to the following formulas:

\begin{theorem}[{\bf Weyl Character Formula in KK-Theory}]
\label{thm-weyl-kk}
\begin{equation}
\label{eq-weyleq1}
[\pt \leftarrow \Flags] \otimes [\wdg^*T^*\Flags]\otimes [\Flags\to \pt] = |W|\cdot [\pt \to \pt] \in \KK_G(\pt,\pt) 
\end{equation}
and
\begin{equation}
\label{eq-weyleq2}
[\Flags \to \pt] \otimes [\pt \leftarrow \Flags] \otimes [\wdg^*T^*\Flags] = \sum_{w\in W} [\Flags\stackrel w \to \Flags ]\in \KK_G(\Flags,\Flags) .
\end{equation}
\end{theorem}

\begin{proof}
The restriction map
\[
 \KK_G(\pt,\pt)\to  \KK_T(\pt,\pt)
 \]
 is injective, and so  we can prove (\ref{eq-weyleq1}) by calculating in $T$-equivariant $\KK$-theory.  Inserting (\ref{eq-F-localization}) into the left hand side of (\ref{eq-weyleq1}) gives
\[
[\pt \leftarrow F]\otimes [F\to \pt] \in \KK_T(\pt,\pt),
\]
which is equal to $|W|\cdot [\pt\to\pt]$, as required.

To prove (\ref{eq-weyleq2}) we shall use the Poincar\'e duality and induction isomorphisms
\[
\KK_G(\Flags,Y)\stackrel \cong \longrightarrow \K_G (Y\times \Flags) 
 \stackrel \cong \longrightarrow \K_T(Y) 
\]
of  (\ref{eq-pd-map}) and (\ref{eq-simple-kk-ind-iso}). 
  Of course we shall set   $Y$ equal to the flag variety. Under Poincar\'e duality the left-hand side of (\ref{eq-weyleq2})  maps to 
\begin{multline}
\label{eq-pd-calc}
 1_\Flags \otimes [\Flags\stackrel\delta\to \Flags\times \Flags]\otimes [\Flags\times \Flags\to \pt\times \Flags] \\
\otimes [\pt\times \Flags\leftarrow \Flags\times \Flags]\otimes
 [ \wdg^*T^*\Flags \boxtimes 1_\Flags] \in \K_G(\Flags\times \Flags) .
 \end{multline}
The composition $\Flags\to \Flags\times \Flags\to \pt \times \Flags$ is the canonical isomorphism, and so
 \[
 1_\Flags \otimes [\Flags\stackrel\delta\to \Flags\times \Flags]\otimes [\Flags\times \Flags\to \pt\times \Flags] = 1_{\pt \times \Flags} \in \K_G(\pt\times \Flags) .
 \]
It follows that (\ref{eq-pd-calc}) is  simply 
\[
  \wdg^*T^*\Flags \boxtimes 1_\Flags  \in \K_G(\Flags\times \Flags),
 \]
 which corresponds under the induction isomorphism to 
$
 \wdg^* T^*\Flags \in \K_T(\Flags) $.
Because of  (\ref{eq-F-localization}), this is equal to
 \[
 1_F \otimes [F\to \Flags]\in \K_T(\Flags),
 \]
 or in other words
 \[
 \sum_{w\in W}  1_\pt \otimes [\pt\stackrel {w} \to \Flags] ,
\]
where the image of the map labeled $w$ is  $\pt\cdot w\in \Flags$.  The individual terms in this sum are precisely the images of the intertwiners under the combined Poincar\'e duality and induction isomorphisms.
 \qed
\end{proof}

\subsection{Comparison with Weyl's Formula}
\label{sec-comparison}

To  interpret Theorem~\ref{thm-weyl-kk} in representation-theoretic terms we need  the vocabulary of roots and weights.  Here is a quick review.  

\begin{definition}
\label{def-flags}
Let $\g$ be the  {complexified}   Lie algebra   of  $G$.  The \emph{flag variety} $\Flags$ is the space    of  Borel subalgebras of $\g$.
\end{definition}

See for example \cite[Chap.\ 3]{MR1433132}.  The flag variety is a nonsingular projective algebraic variety, and in particular a closed complex manifold.  The compact group  $G$ acts transitively on $\Flags$ and the stabilizer of any point in $\Flags$  is a maximal torus $T\subseteq G$. So the choice of a $T$-fixed basepoint in $\Flags$ identifies the flag variety with the homogeneous space  $G/T$.

\begin{remark} These identifications  of $G/T$ with $\Flags$ yield  all the $G$-equivariant complex structures on $G/T$, although there are other  $G$-equivariant almost-complex structures.
\end{remark}

 Fix a basepoint in $\Flags$, that is to say a Borel subalgebra $\b\subseteq \g$, and let $T\subseteq G$ be the maximal torus that fixes $\b$.  Then $\b = \h\oplus \n$, where $\h$ is the complexification of the Lie algebra of $T$ and $\n= [ \b,\b]$. Moreover
\[
\g   = \overline{\n} \oplus \h \oplus \n .
\]
The vector space $\overline{\n}$, which is a representation of $T$, decomposes into $T$-eigenspaces, each of dimension one.  The corresponding eigenvalues are characters of $T$ and may be written 
in the form
\begin{equation}
\label{eq-weight-vs-character}
e(\alpha)\colon \exp(v)\mapsto \exp (\alpha(v)) ,
\end{equation}
where $v$ is in the Lie algebra of $T$ and  the  complex-linear forms $\alpha\in \h^*$ are by definition  the \emph{positive roots} of $G$ (with respect to the given choice of $\b$;  our conventions follow, for example, the book of Chriss and Ginzburg  \cite[Section 3.1]{MR1433132}).    The positive roots are examples of \emph{integral weights}, meaning elements   $\varphi \in \h^*$ that exponentiate, as in (\ref{eq-weight-vs-character}), to characters $e(\varphi)$ of $T$.

 Under the isomorphism 
$
\K_G(\Flags) \cong \K_T(\pt)$
given by restriction to the basepoint $\pt\in \Flags$, the $\K$-theory class of the bundle  $\wdg^*T^*\Flags$ maps to  \[
\Lambda = \prod_{\alpha> 0} (1- e(-\alpha)),
\]
where the product is over the positive roots (we are writing elements of $\K_T(\pt)$ as characters, that is to say as functions on $T$). This is because the tangent space to $\Flags$ at the basepoint identifies with $\overline\n$, and hence the dual vector space identifies with $\n$ as a representation of $T$.

The Weyl group $W\cong N(T)/T$ acts via the adjoint representation on $T$ and its Lie algebra, and therefore on  the set of all integral weights in the usual way:
\[
w(\psi)(X) = \psi (w^{-1}(X)).
\]
  The positive roots determine a partial order on the set of integral weights, and in each $W$-orbit there is a maximum element, called a \emph{dominant integral weight}.

The Weyl group also acts on $K_T(\pt)\cong R(T)$ via the formula 
\[
w(e(\psi)) = e(w(\psi)),
\]
 and  it is a basic fact that 
\begin{equation}
\label{eq-w-action-lambda}
w(\Lambda) =(-1)^w  e( \rho-w(\rho))\cdot \Lambda,
\end{equation}
for all $w\in W$, where as usual
\[
\rho = \tfrac 12 \sum _{\alpha>0} \alpha\in \h^*,
\]
and where the sign is the determinant of $w$ as it acts on the Lie algebra of $T$.

It follows from Lemma~\ref{lemma-localization2} that the intertwining operators $I_w$ fix the class $\wdg^*T^*\Flags\in K_G(\Flags) $.
From this and (\ref{eq-w-action-lambda}) it is straightforward to check  that the intertwining operators  $I_w$ act as follows on $\K_G(\Flags)\cong \K_T(\pt)$:
\begin{equation}
\label{eq-intertwiner-fmla}
e(\psi)\otimes I_w = (-1)^we(w^{-1}(\psi+\rho)-\rho)\in \K_T(\pt),
\end{equation}
\begin{remark}
\label{remark-atiyah-embedding}
Using these things we can complete some unfinished business.  We want to show that 
\[
[\pt \leftarrow \Flags]\otimes [\Flags\to \pt] = \id_\pt \in \KK_G(\pt,\pt)
\]
as in (\ref{eq-atiyah-embedding}).  To calculate the left-hand side we can     restrict to $T$ since restriction here is injective; call the result $\Phi\in \KK_T(\pt,\pt)$.  According to Theorem~\ref{thm-weyl-kk},
\[
\Phi\cdot  \prod_{\alpha >  0} (1- e(-\alpha))  = \sum (-1)^w e(w(\rho)-\rho) .
\]
Comparing the two sides, we see that  the only dominant weight occurring in $\Phi$ is the trivial weight. Since   the weights of $\Phi$ are acted on by $W$, the only weight of $\Phi$ is the trivial weight, and it occurs with multiplicity one.  Hence $\Phi$ is the trivial representation, as required.  See \cite{MR0228000} for a different, geometric approach.
\end{remark}

\begin{remark}
Let $V$ be a finite-dimensional representation of $G$.  The localization map associates to $V$   (the $\K$-theory class of) the complex of vector bundles 
\[
\wdg^{0} \n\otimes V \longleftarrow \wdg^{1} \n\otimes V \longleftarrow \cdots  \]
that fiberwise computes nilpotent Lie algebra homology (here $d=\dim (\n)$, or
 in other words the number of positive roots).  So the  localization map could equally well be defined to map $V$ to the alternating sum of the  bundles whose fibers are the homology spaces $H_p(\n,V)$. The connection between the Weyl character formula and nilpotent Lie algebra homology was thoroughly explored by Kostant in  \cite{MR0142696}. \end{remark}

By now we have encountered all the ingredients of Weyl's formula:

\begin{theorem}[Weyl Character Formula]
\label{thm-weyl}
There is a bijection between equivalence classes of irreducible finite-dimensional representations of $G$ and dominant integral weights under which the character of the representation associated to a dominant integral weight $\psi$ is equal to
\[
\frac{\sum_{w\in W} (-1)^w e(w(\psi + \rho)-\rho)}{    \prod_{\alpha >  0} (1- e(-\alpha))} 
\]
on $T$ \textup{(}the product is over the positive roots\textup{)}.  The highest weight of the representation is $\psi$. 
\end{theorem}

Let us write 
\begin{equation}
\label{eq-lambda-def}
\Lambda \colon \K_G (\pt) \longrightarrow  \K_G (\Flags )
\end{equation}
for the map induced by pullback from a point, followed by multiplication against $\wdg^* T^*\Flags \in \K_G (\Flags)$,  and 
\begin{equation}
\label{eq-gamma-def}
\Gamma \colon \K_G (\Flags ) \longrightarrow \K_G (\pt)
\end{equation}
for the wrong-way map induced from the collapse of $\Flags$ to a point.  Theorem~\ref{thm-weyl-kk} asserts that 
\begin{equation}
\label{eq-k-weyl1}
\Lambda\circ \Gamma = \sum_{I_w} I_w\colon \K_G (\Flags) \longrightarrow \K_G (\Flags)
\end{equation}
and 
\begin{equation}
\label{eq-k-weyl2}
\Gamma\circ \Lambda = |W|\cdot \id \colon \K_G(\pt) \longrightarrow \K_G(\pt) .
\end{equation}
According to equation~(\ref{eq-k-weyl1}), if $\psi$ is any integral weight, then the character of the virtual representation 
\[
V_\psi=\Gamma (e(\psi))\in \K_G(\pt)
\]
 is given by Weyl's formula, and it follows from the formula that if $\psi$ is dominant, then the highest weight of the representation is $\psi$.  The rest of the content of Theorem~\ref{thm-weyl}, that these virtual representations are actually irreducible representations, and that they are all of them, requires a bit more effort, going beyond $\K$-theory and our  concerns here.\footnote{A more thorough analysis of the combinatorial relations between the Weyl group  and the positive roots   shows  that the virtual representations $V_\psi$ associated to dominant weights form a basis for $\K_G(\pt)$.  Meanwhile the  Weyl integral formula implies that $\Gamma$ and $\Lambda$ are adjoint with respect to the natural inner products on $\K_G(\pt)$ and $\K_G(\Flags)\cong \K_T(\pt)$ for which the irreducible representations are orthonormal:
\begin{equation*}
\label{eq-adjoint-rel}
\bigl\langle x, \Gamma(z) \bigr\rangle_G
=
\bigl\langle \Lambda(x),   z\bigr\rangle_T
\end{equation*}
From this it follows easily that $\Gamma$ maps dominant integral weights to irreducible representations.
 }

 \subsection{A Second Look at KK-Theory for the Flag Variety}
 \label{sec-bott-sam}

We shall continue to borrow from Lie theory and  sketch a calculation of the ring $\KK_G(\Flags,\Flags)$ (which we shall not use later).  In view of the Poincar\'e duality isomorphism, this is effectively a problem in equivariant topological $\K$-theory, and indeed it is one that has been addressed and solved in this guise (for example the discussion here closely parallels \cite{MR2653685} in many places).  

Let $\alpha$ be a simple positive root.  Associated to it there is a \emph{generalized flag variety} $\Flags_\alpha$ consisting of all subalgebras $\mathfrak p\subseteq \g$ conjugate to 
\[
\mathfrak p_\alpha = \overline \n_\alpha \oplus \b ,
\]
where $\overline \n_\alpha$ is the $\alpha$-root space in $\overline \n$.  There is a a $G$-equivariant submersion $ \Flags \to \Flags_\alpha$ given by inclusion of subalgebras $\b\subseteq \mathfrak p$. 
 Define
\begin{equation}
\label{eqn-bott-samelson-proj}
\Pi_\alpha = [\Flags \to \Flags_\alpha]\otimes [\Flags_\alpha \leftarrow \Flags] \in \KK_G(\Flags, \Flags).
\end{equation}

\begin{proposition}
\label{prop-bs-idempotent}
If $\alpha$ is any simple root, then 
\[
\Pi_\alpha \otimes \Pi_\alpha = \Pi _\alpha \in \KK_G(\Flags,\Flags).
\]
\end{proposition}

\begin{proof}
We want to show that 
\[
 [\Flags \to \Flags_\alpha]\otimes [\Flags_\alpha \leftarrow \Flags]
 \otimes  [\Flags \to \Flags_\alpha]\otimes [\Flags_\alpha \leftarrow \Flags] =\ [\Flags \to \Flags_\alpha]\otimes [\Flags_\alpha \leftarrow \Flags],
 \]
and so of course it suffices to show that 
\begin{equation}
\label{eq-bs-proj1}
 [\Flags_\alpha\leftarrow  \Flags ]\otimes [\Flags\to \Flags_\alpha] = \id_{\Flags_\alpha}\in \KK_G(\Flags_\alpha, \Flags_\alpha).
 \end{equation}
But if $G_\alpha\subseteq G$ is the stabilizer of the basepoint in $\Flags_\alpha$, then
$ \Flags \cong \Ind_{G_\alpha}^G G_\alpha/T $,
and   (\ref{eq-bs-proj1}) follows   from the identity  
\[
[\pt \leftarrow G_\alpha/T]\otimes [G_\alpha/T\to \pt] =\id_\pt \in \KK_{G_\alpha}(\pt,\pt) 
\]
using the induction homomorphism (\ref{eq-ind-map-kk}). The latter identity  is Atiyah's formula (\ref{eq-atiyah-embedding})  for the group $G_\alpha$. \qed
\end{proof}

The idempotents $\Pi_\alpha$ are related to intertwiners, as follows:

\begin{proposition}
\label{prop-divided-diff}
If $\alpha$ is a simple root, then 
\[
\Pi_\alpha\otimes [1-e(-\alpha)]   = \Id_{\Flags} + I_{s_\alpha} \in \KK_G(\Flags,\Flags),
\]
where $s_\alpha\in W$ is the reflection associated to $\alpha$.
\end{proposition}

\begin{proof} If $G=G_\alpha$, then this is the Weyl character formula.  The general case follows by induction from $G_\alpha$ to $G$ as above. 
\qed
\end{proof}

The $\Pi_\alpha$, at least when considered as operators on $\K_G(\Flags)$ by Kasparov product, are called \emph{Demazure operators}. Compare \cite[Sec.\ 1]{MR0430001} or \cite[Sec.\ 1]{MR2653685}.  
From the proposition we obtain the following formula for the action of $\Pi_\alpha$ on $\K_G(\Flags)$:
\begin{equation}
\label{eq-div-diff-ops1}
e(\phi)\otimes \Pi_\alpha  = \frac{e(\phi)-e(-\alpha)e(s_\alpha(\phi))}{1-e(-\alpha)}\in \K_G(\Flags).
\end{equation}
In addition the commutation relations in $\KK_G(\Flags, \Flags)$ between the elements $\Pi_\alpha$ and elements of the subring $\K_G(\Flags)\subseteq \KK_G(\Flags,\Flags)$ are as follows:
\begin{equation}
\label{eq-div-diff-ops2}
[e(\phi)]\otimes \Pi_\alpha = \Bigl [ \frac{e(\phi)-e(s_\alpha(\phi))}{1-e(-\alpha)}\Bigr] + \Pi_\alpha\otimes [e(s_\alpha(\phi))]\in \KK_G(\Flags,\Flags).
\end{equation}

\begin{remark} 
In the reverse direction, Proposition~\ref{prop-divided-diff} gives a new formula for the intertwining operators.  It is of some interest because its constituent parts are maps and bundles that are holomorphic and equivariant for the full group of symmetries of $\Flags$, namely the complexification $G_\C$ of $G$ (in contrast to the $G$-equivariant maps from $\Flags$ to itself that we considered in Section~\ref{sec-intertwiners}).  

Another approach to intertwiners uses the Bruhat cells in $\Flags\times \Flags$.  These are the $G_\C$-orbits, and each one contains the graph of a unique $G$-equivariant map from $\Flags$ to itself.  So we may label them as 
\begin{equation}
\label{eq-bruhat}
\Cell^w\subseteq \Flags\times \Flags,
\end{equation}
where $w\in W$ (recall that  we initially defined the Weyl group to be the group of $G$-equivariant self-maps of the flag variety).  The $\K_G$-theory wrong-way maps associated to the two coordinate projections from $\Cell^w$ to $\Flags$  are isomorphisms. Indeed both projections  give $\Cell^w$  the structure of a complex $G$-equivariant vector bundle (the zero section is the graph of $w$), and so the Thom isomorphism theorem applies. Composing one with the inverse of the other gives an intertwiner.  The commuting diagram
\[
\xymatrix@C=10pt@R=8pt{
&\Cell^w\ar[dl]\ar[dr]&\\
 \Flags\ar[rr]_w&&\Flags ,
 }
 \]
where the left and right downwards maps are the left and right projections, shows that the intertwiner defined this way is the same as ours.  Compare \cite{EY1}.
\end{remark}
  
  Returning to the idempotents $\Pi_\alpha$, we next consider how the idempotents associated to two distinct simple roots interact.
  
\begin{proposition}[Demazure \cite{MR0430001}]
The operators $\Pi_\alpha\in \KK_G(\Flags,\Flags)$ associated to simple roots satisfy the braid relations.
\end{proposition}

\begin{proof}
The braid relations may be checked by computation  using Proposition~\ref{prop-divided-diff} (although as Demazure writes, the calculation  ``est clair si $m=2$, facile pour $m=3$, faisable pour $m=4$ et \'epouvantable pour $m=6$").  See also \cite{MR1044959,MR957070} for some improvements on the brute force approach.\qed
\end{proof}

\begin{proposition}
Let $ \alpha_1,\dots, \alpha_n$ be a reduced list of simple roots.  
The class 
\begin{equation}
\label{eq-def-pi-w}
\Pi_{\alpha_1}\otimes\cdots \otimes \Pi_{ \alpha_n}\in \KK_G(\Flags,\Flags)
\end{equation}
 depends only on the element $w=s_{\alpha_1}\cdots s_{\alpha_n}\in W$.  The subring of $\KK_G(\Flags,\Flags)$ generated by the elements $\Pi_\alpha$ associated to simple roots is the universal ring generated by idempotents $\Pi_\alpha$ subject only to the braid relations, and it is spanned  over $\Z$  by the elements \textup{(}\ref{eq-def-pi-w}\textup{)}.
\end{proposition}

\begin{proof}
See \cite[Chap~4, Sec~1, Prop~5]{MR1890629}.  \qed
\end{proof}

\begin{definition}
If $w\in W$, then denote  by $\Pi_w\in \KK_G(\Flags,\Flags)$ the Kasparov product  \textup{(}\ref{eq-def-pi-w}\textup{)}.
\end{definition}

These elements   constitute a basis for $\KK_G(\Flags,\Flags)$ (as either a left or right module over $\K_G(\Flags)$).  We can  see this by making use of the following additional geometric ideas.

\begin{definition}
Let $\alpha   =(\alpha_1,\dots, \alpha_n)$ be an ordered list of simple roots (with repetitions or reoccurrences allowed). Define
\begin{multline}
\label{eq-bs-def}
\BS^{\alpha}   =  \bigl \{\, (w_0,\dots  ,w_n) \in \Flags \times \cdots \times \Flags: \, \pi_{\alpha_k}(w_{k-1}) = \pi_{\alpha_k}(w_{k})  \quad \forall k   \,\bigr \} ,
\end{multline}
where $\pi_{\alpha_k}$ is the projection from $\Flags$ onto $\Flags_{\alpha_k}$.
In addition, denote by  
\[
p_0,p_n\colon \BS^\alpha\longrightarrow \Flags 
\]
the projections onto the $w_0$- and $w_n$-coordinates.
\end{definition}

\begin{remark} 
\label{rem-bs-rem}
The 
 \emph{Bott-Samelson variety} $\BS ^\alpha _{\pt}$ associated to the list of simple roots $\alpha = (\alpha_1,\dots, \alpha_n)$ and to a given basepoint $\pt \in \Flags$ is  the inverse image of the basepoint  under $p_n$; see  \cite{MR0105694,MR0376703}.   This explains the notation.
\end{remark}

\begin{definition}
Let $ \alpha = (\alpha_1,\dots , \alpha_n)$ be any ordered list of simple roots. Denote by 
$
\Pi_{\alpha_1,\dots,\alpha_n}\in \KK_G(\Flags,\Flags)$ the Kasparov product
\[ 
[\Flags\stackrel{p_0} \longleftarrow  \BS^\alpha ] 
\otimes [\BS^\alpha\stackrel{p_n}\longrightarrow \Flags] .
\]
\end{definition}

\begin{example}
\label{ex-bs-length1}
If the list has length one, then the new definition of $\Pi_\alpha$ agrees with the old one in view of the pullback diagram
\[
\xymatrix@C=12pt@R=12pt{\BS^\alpha\ar[r]^{p_1} \ar[d]_{p_0}&\Flags \ar[d] \\
\Flags \ar[r] & \Flags_\alpha 
}
\]
and Lemma~\ref{lemma-base-change}.
\end{example}

\begin{lemma} If $\alpha_1,\dots, \alpha_n$ is any list of simple roots, then
\begin{equation*}
\Pi_{\alpha_1,\dots, \alpha_n}
 = \Pi_{\alpha_1}\otimes \cdots \otimes \Pi_{\alpha_n}\in \KK_G(\Flags,\Flags).
\end{equation*}
\end{lemma}

\begin{proof}
This follows by repeated application of the base-change formula in Lemma~\ref{lemma-base-change}.
\qed
\end{proof}

 Now, fixing a basepoint $\pt\in \Flags$ as usual, fixed by $T\subseteq G$, define 
  \[
 \Cell^w_\pt = \{\, \frak b \in \Flags \, : \, (\frak b, \pt)\in \Cell^w\,\}.
 \]
   The $\Cell^w_\pt$ are the \emph{Bruhat cells} in $\Flags$.  They are affine $T$-spaces (of dimension equal to the length of $w$), they are pairwise disjoint, their union is $\Flags$, and the closure of each $\Cell^w_\pt$  is contained in the union of cells of lower dimension. This is the \emph{Bruhat decomposition}.  Using it, we find  that:

\begin{lemma}
The   group $\KK_G(\Flags,\Flags)$ is a free right and left module  of rank $|W|$ over the commutative ring $\K_G(\Flags)$.  \qed
\end{lemma}

\begin{lemma} The   elements
$\Pi_{w}\in \KK_G(\Flags,\Flags)$ freely generate   $\KK_G(\Flags,\Flags)$ as a right or left $\K_G(\Flags)$-module, as $w$ ranges over the Weyl group.
\end{lemma}

 \begin{proof} Let us consider the left-module result; the right module result is similar and in any case it can be deduced from the left module result by means of \eqref{eq-div-diff-ops2}.  Let $w\in W$, let 
 \[
 w = s_{\alpha_1}\cdots s_{\alpha_n}
 \]
 be a factorization of $w$ as a minimal-length product of simple reflections and let $\alpha=(\alpha_1,\dots, \alpha_n)$.  Under the induction/restriction isomorphism 
 \[
 \KK_G(\Flags,\Flags) \stackrel \cong \longrightarrow \KK_T(\Flags, \pt)
 \]
 the element $\Pi^w = \Pi^{\alpha}$ maps to 
 \begin{equation}
 \label{eq-image-bs}
 [\Flags \leftarrow \BS^{\alpha}_\pt] \otimes [\BS^{\alpha}_\pt \to \pt]
 \end{equation}
 Now the image of $\BS^\alpha_\pt$ in $\Flags$ is the closure of the Bruhat cell $\Cell^w_\pt$ (see \cite{MR0376703}), so the class \eqref{eq-image-bs} actually lies in the image of 
 \[
 \KK_T(\overline{\Cell^w_\pt}, \pt) \longrightarrow \KK_T(\Flags, \pt)
 \]
 Moreover the projection from the inverse image of $\Cell^w_\pt$ in $\BS^\alpha_\pt$ back down to $\Cell^w_\pt$ is an isomorphism (see \cite{MR0376703} again). So under  
 \[
  \KK_T(\overline{\Cell^w_\pt}, \pt) \longrightarrow  \KK_T( \Cell^w_\pt , \pt)  
 \]
 the class \eqref{eq-image-bs} is mapped to a generator.  The lemma follows from this and the Bruhat decomposition.
 \qed
 \end{proof}

Finally, let us mention Demazure's character formula  \cite{MR0430001,MR778124,MR782239} and its relation to the problem of computing the restriction map in equivariant $\K$-theory.

\begin{theorem}  If $\Pi\in \KK_G(\Flags,\Flags)$ is the class $[\Flags\to\pt]\otimes[\pt\leftarrow \Flags]$ from \textup{(}\ref{eq-def-of-pi}\textup{)}, then 
\[
	\Pi = \Pi_{\alpha_1,\dots, \alpha_n} = \Pi_{\alpha_1}\otimes \cdots \otimes \Pi_{\alpha_n}
	\in \KK_G(\Flags,\Flags),
\]
where $\alpha_1,\dots,\alpha_n$ is any list of simple roots for which the associated product of simple reflections is the longest element of the Weyl group.
\end{theorem}

As pointed out in \cite{MR2653685}, Demazure's formula gives a characterization of the range of the restriction map 
\[
\K_G(Z)\longrightarrow \K_T(Z),
\]
for any $G$-space $Z$, or equivalently of the range of the pullback map
\begin{equation}
\label{eq-general-restriction}
\K_G(Z)\longrightarrow \K_G(\Flags\times Z).
\end{equation}

\begin{lemma} The right ideal in $\KK_G(\Flags,\Flags)$ generated by the $\Id - \Pi_\alpha$ as $\alpha$ ranges over the simple roots is equal to the right ideal generated by  $\Id - \Pi$.
\end{lemma}

\begin{proof}
If $\alpha_1,\dots, \alpha_k$ is  {any} list of simple roots (of any length), then 
\[
	\Id  - \Pi_{\alpha_1,\dots, \alpha_k}= ( \Id - \Pi_{\alpha_k}) + ( \Id - \Pi_{\alpha_1,\dots, \alpha_{k-1}} )\otimes \Pi_{\alpha_k} 
\]
and so by an induction argument $\Id-\Pi_{\alpha_1,\dots, \alpha_k}$ is in the right-ideal of $\KK_G(\Flags,\Flags)$ generated by the $\Id-\Pi_{\alpha}$.   In particular, $\Id - \Pi$ lies in this ideal.

Conversely, we have 
\[
	\Id - \Pi_{\alpha_n}  =  ( \Id - \Pi )  -  ( \Id - \Pi )\Pi_{\alpha_n}
\]
since $\Pi_{\alpha_n}$ is an idempotent.  But any simple root $\alpha$ can play the role of $\alpha_n$ in the formula for $\Pi$ (that is, it can appear last in some reduced expression for the longest word in $W$).  We find therefore that if $\alpha$ is any simple root then $\Id -\Pi_{\alpha}$ lies in the right ideal generated by $\Id - \Pi$.    
\qed
\end{proof}

The range of (\ref{eq-general-restriction}) is equal to the range of the continuous transformation
\[
\Pi\colon K_G(\Flags\times Z)\longrightarrow \K_G(\Flags\times Z),
\]
and since $\Pi$ is an idempotent, this is the same as the kernel of $\Id-\Pi$.  The lemma implies that this is the same as the joint kernel of the operators $\Id - \Pi_\alpha$.

See also \cite{leung2011}, where the argument of \cite{MR2653685} is also considered within the context of  $\K\K$-theory, although not   as we did  above.

\section{The  Baum-Connes Conjecture}
\label{sec-bc}

In this final section we shall examine the Baum-Connes conjecture  \cite{MR1292018} for connected Lie groups in the light of the previous calculations.    We shall not go further here than reinterpret the Baum-Connes conjecture (for a reductive Lie group, and with commutative coefficients) in terms of equivariant $K$-theory for the flag variety and a Weyl-type character formula.  But we hope that eventually this may shed some light on strategies for proving the conjecture (which remains open even in the generality we are considering).  We would also like to use the reformulation as a first step towards Baum-Connes isomorphisms with those of geometric representation theory.  However these are topics for future work.

\subsection{Equivariant K-Theory for Noncompact Groups}
\label{sec-k-noncpt-g}

Let $G$ be a locally compact Hausdorff topological group and let $X$ be a locally compact Hausdorff space equipped with a continuous action of $G$.  There is more than one plausible definition of the $G$-equivariant $K$-theory of $X$.  We shall use the definition that is most closely related to the Baum-Connes conjecture. It requires a certain amount of functional and harmonic analysis.

Fix a left Haar measure on $G$.  The \emph{reduced $C^*$-algebra} of $G$, denoted here $C^*_\red (G)$, is the completion in the operator norm of the convolution algebra of continuous, compactly supported functions on $G$ as it acts on the Hilbert space $L^2(G)$.  See for example \cite{MR548006}.  One reason for preferring this particular Banach algebra completion is its close connection to Fourier theory: if $G$ is abelian, then $C^*_\red (G)$ is isometrically isomorphic via the Fourier transform to the algebra of continuous complex functions, vanishing at infinity, on the Pontrjagin dual of $G$.  Thus $C^*_\red(G)$ is relevant to problems where the \emph{topology} of the dual space is important.  

For nonamenable groups, $C^*_\red(G)$ reflects topological features of the \emph{reduced} unitary dual.  For reductive Lie groups this is the support of the Plancherel measure, or in other words the tempered unitary dual.

If $A$ is any $C^*$-algebra that is equipped with a continuous action of $G$ by $C^*$-algebra automorphisms, then the \emph{reduced crossed product algebra} $C^*_\red(G,A)$  is similarly a completion of the twisted convolution algebra of compactly supported and continuous functions from $G$ into $A$.  The convolution product is 
\[
f_1\star f_2 (g) = \int _G f_1(h) \alpha_h(f_2(h^{-1}g))\, dh,
\]
where $\alpha$ denotes the action of $G$ on $A$.  If $A$ is represented faithfully and isometrically on a Hilbert space $H$, then $C^*_\red(G,A)$ is represented faithfully and  isometrically on the Hilbert space completion $L^2(G,H)$ of the continuous and compactly supported functions from $G$ into $H$.  See \cite{MR548006} again.

We define the equivariant $\K$-theory of $A$ to be the $\K$-theory of the reduced crossed product algebra:
\begin{equation}
\label{eq-eq-k-theory1}
\K_G(A) = \K(C^*_{\red}(G,A)),
\end{equation}
and if $X$ is a locally compact $G$-space, then we define the equivariant $\K$-theory of $X$ to be the $\K$-theory of the reduced crossed product algebra in the case $A=C_0(X)$:
\begin{equation}
\label{eq-eq-k-theory2}
\K_G(X) = \K(C^*_{\red}(G,C_0(X))) .
\end{equation}
The definition is arranged so that $\K_G(\pt) = \K(C^*_\red(G))$.  As we noted,  this $K$-theory group should reflect aspects of the topology of the tempered dual,\footnote{The exemplar is the discrete series, which appear as isolated points in the tempered dual and are reflected simply and clearly in the $\K$-theory.  See the introduction.}  and the aim of the Baum-Connes theory is to use geometric techniques to understand these aspects more fully.

Unfortunately in most cases it is not easy to give a geometric description of cycles and generators for equivariant $\K$-theory akin to the standard vector bundle description of $\K$-theory in the case of compact groups.  Because of this it is considerably more difficult to define geometric operations on $\K$-theory groups (such as wrong-way maps, and so on) in the noncompact case.  It is here that Kasparov's equivariant $\K\K$-theory plays a crucial role, providing by means of functional-analytic constructions some operations in some circumstances. We shall invoke it in some very simple situations. As we shall see, the Baum-Connes conjecture asserts that a full suite of operations can   indeed be obtained by reduction to the case of compact groups, but of course it is at the moment still a conjecture.

The equivariant $\K$-theory group $\K_G(X)$ that we have defined is not in general a  ring.  It is however a module over the Grothendieck ring  $\K^\vb_G(X)$ of $G$-equivariant Hermitian vector bundles on $X$. In terms of Kasparov's analytic theory this may be explained as follows.  Kasparov defines a product  
\begin{equation}
\label{eq-kk-action}
\K_G(X)\otimes \K\K_G(X,X)\longrightarrow \K_G(X)
\end{equation}
in \cite[Sec 3.11]{MR918241}. Note that this involves Kasparov's functional-analytic $\K\K$-theory, not ours. Meanwhile there is a ring homomorphism 
\begin{equation}
\label{eq-kasp-kk-classes}
\K_G^\vb(X)\longrightarrow \K\K_G(X,X) .
\end{equation}
It is a counterpart of  the one defined in (\ref{eq-vb-class2}).  The combination of (\ref{eq-kk-action}) and (\ref{eq-kasp-kk-classes}) gives our the required module action
\begin{equation}
\label{eq-vb-action}
\K_G(X)\otimes \K_G^\vb(X)\longrightarrow \K_G(X) .
\end{equation}

There is a modest but nonetheless very important extension of (\ref{eq-vb-action}).  Suppose given a bounded and compactly supported complex of equivariant hermitian vector bundles
\begin{equation}
\label{eq-almost-eq-cplx}
E_0\longleftarrow E_1 \longleftarrow \cdots \longleftarrow E_k
\end{equation}
over a locally compact space $X$, as in (\ref{eq-cplx-over-w}), except that we do \emph{not} require the differentials to be $G$-equivariant.  Instead we require that the conjugate of a differential by an element   $g\in G$ differs from the differential by a bundle homomorphism that vanishes at infinity on $W$ (and varies continuously with $g$).     Each such \emph{almost equivariant Hermitian complex} also determines a class in Kasparov's $\K\K_G(X,X)$, and so acts as an operator on $\K_G(X)$.

Suppose now that $H$ is a closed subgroup of $G$.  There is in general no restriction homomorphism from $G$-equivariant to $H$-equivariant $\K$-theory as there was in the compact case, but there is however an 
induction isomorphism
\begin{equation}
\K_G(\Ind_H^G  X )  \cong \K_H( X)
\end{equation}
exactly analogous to (\ref{eq-ind-map}). 
This follows from the fact that $C^*_\red(G, \Ind_H^G X)$ and $C^*_\red (H,X)$ are in a canonical way Morita equivalent $C^*$-algebras.  More generally, if a locally compact space $W$ admits commuting free and proper actions of two  locally compact groups $H$ and $G$, then there   is a Morita equivalence between  crossed product algebras 
\[
C^*_\red(G, W/H)\quad \text{and}\quad C^*_\red ( H, W/G)
\]
See \cite{MR679709}.\footnote{The cited paper deals with \emph{maximal} crossed product algebras rather than the reduced crossed product algebras that we are considering here.  However we shall only use the theorem in situations where the two types of crossed products agree.}   This leads to a $\K$-theory isomorphism 
\[
\K_G(W/H)\cong \K_H(W/G) .
\]
   In the present case, we may apply this to $W= G\times X$.

 Finally, although equivariant $\K$-theory as we have defined it in this section is in general rather complicated, in the case of a compact group it reduces to the usual equivariant $\K$-theory that we considered earlier in the paper.  This is the content of the  Green-Julg theorem \cite{green,MR625361}.  

\subsection{Equivariant KK-Theory for Noncompact Groups}

We have made use of Kasparov's $\K\K$-theory in the previous section, although only as a shortcut to defining the action of hermitian equivariant vector bundles, or almost equivariant complexes of hermitian equivariant vector bundles, on our $C^*$-algebraic equivariant $\K$-theory.  Now we shall define a version of the $\KK$-theory that we considered for compact groups earlier in the paper.

\begin{definition} Denote by $\KK_G(X,Y)$ the abelian group of all natural transformations
\[
T_Z\colon \K_G(X\times Z)\longrightarrow \K_G(Y\times Z)
\]
with the property that diagrams of the form 
\[
\xymatrix{
\K_G(X\times Z) \ar[r]^{T_Z}\ar[d] & \K_G(Y\times Z)\ar[d] \\
\K_G(X\times Z\times W) \ar[r]_{T_{Z\times W}} & \K_G(Y\times Z\times W) ,
}
\]
in with the vertical maps are multiplication by some equivariant Hermitian vector bundle  or almost-equivariant complexes of Hermitian vector bundles on $W$, all commute.
 \end{definition}

\begin{remark}  Of course if $G$ is compact then this is the very same theory that we considered earlier, and hence it matches closely with Kasparov's theory as described in Section~\ref{sec-kasp-theory}.   In the noncompact case the difference between $\KK_G(X,Y)$ and $\K\K_G(X,Y)$ is usually quite a bit greater. There is a natural map 
\[
\K\K_G(X,Y)\longrightarrow \KK_G(X,Y)
\]
that is compatible with Kasparov products and actions on equivariant $\K$-theory, but  unfortunately it is frequently \emph{not} an isomorphism.  This points to the difficulty of using Kasparov's concretely defined equivariant $\K\K$-theory groups in computations of equivariant $\K$-theory, since it can and does happen that the action of a $\K\K_G$-class on $\K_G$-groups is an isomorphism, whereas the class is not itself an isomorphism in $\K\K_G$-theory.   (On the other hand $\KK_G$-theory is defined directly in terms of equivariant $\K_G$-theory, but its abstract definition reduces its value as a computational tool.)
\end{remark}

 \subsection{Reformulation of the Baum-Connes Conjecture}
\label{sec-bc-reformulation}

From now on assume that $G$ is a connected Lie group.  It contains a maximal compact subgroup $K$, and $K$ is unique up to conjugacy \cite[Chapter VII]{MR1661166}.
 Form the 
``symmetric space'' $S=G/K$ (we're interested in the case where $G$ is a reductive group, in which case $S$ really is a Riemannian symmetric space).

 In what follows in the final part of this paper we shall focus on the \emph{restriction homomorphism} 
\begin{equation}
\label{eq-kk-restr2}
\KK_G(X,Y)\longrightarrow \KK_K (X,Y).
\end{equation}
This is defined exactly as in (\ref{eq-ind-homo2}); as noted in Remark~\ref{rem-restr-in-kk}, the definition uses only  the induction isomorphism in equivariant $\K$-theory, which is still available to us in the noncompact context, and not   restriction, which isn't.

Baum and Connes define an \emph{assembly map}, involving Kasparov's $\K\K$-theory, as follows:
\[
\mu\colon \K\K_G (S, X)\longrightarrow \K_G(X)
\]
See \cite[Sec.\ 3]{MR1292018} (the assembly map is in fact defined for all $G$-$C^*$-algebras $A$, but we shall keep to the main case of interest, which is $A=C_0(X)$).

\begin{conjecture}[Baum and Connes]
\label{conj-bc}
Let $G$ be a connected Lie group with maximal compact subgroup $K$, and let $S=G/K$.  If $X$  is any  locally compact $G$-space, then 
the assembly map
\begin{equation}
\label{eqn-bc-formula}
\mu\colon \K\K_G (S, X)\longrightarrow \K_G(X)
\end{equation}
is an isomorphism of abelian groups.
\end{conjecture}

\begin{remark}
This is what the experts would call the Baum-Connes conjecture with commutative coefficients---the coefficients being the $C^*$-algebra $C_0(X)$.  The case $X=\pt$ is known for all connected Lie groups; see \cite{MR2010742}.  The conjecture as stated is known for a limited class of groups---among the  noncompact simple Lie groups only those of rank one \cite{MR741223,MR1332908,MR1903759}.
\end{remark}

Now Kasparov showed that there is  a Poincar\'e duality isomorphism
\[
 \K\K_G (S, X)\stackrel\cong\longrightarrow \K_G (X\times  T S),
 \]
  and so the Baum-Connes assembly map can be viewed as  a homomorphism
 \begin{equation}
 \label{eq-pd-form-bc}
  \K_G(X\times TS) \longrightarrow \K_G (X) .
  \end{equation}
It is  implemented by a class   $D\in \K\K_G (T S, \pt)$, namely Kasparov's \emph{Dirac class} \cite[Sec.~4]{MR918241}.  It is a remarkable feature of the Kasparov theory that there is a class $ D{\check{\phantom{i}}}\in \K\K_G(\pt, TS)$, the \emph{dual Dirac class}, which induces a right inverse map
 \begin{equation}
 \label{eq-dual-dirac-bc}
   \K_G(X) \longrightarrow \K_G (X\times TS) 
 \end{equation}
to (\ref{eq-pd-form-bc}); see \cite[Sec.~5]{MR918241}.  Thus
  \[
D\otimes   D {\check{\phantom{i}}}= \Id_{TS}\in \K\K_G(TS,TS).
\] 
 The dual Dirac class is represented by an almost-equivariant complex of Hermtian vector bundles on $TS$. 
   
\begin{lemma}
\label{lemma-equiv-bc}
Conjecture~\ref{conj-bc} is equivalent to the assertion that the restriction map
\[
R=\Res_K^G \colon \KK_G(X,Y)\longrightarrow \KK_K(X,Y)
\]
is an isomorphism for all locally compact $G$-spaces $X$ and $Y$.
\end{lemma}
 
\begin{proof}
If $X$ and $Y$ are locally compact $G$-spaces, and if $T\in \KK_G(X,Y)$, then its restriction to a class $R( T)\in \KK_K(X,Y)$ is given by the diagram
\begin{equation*}
\xymatrix{
\K_K (X\times Z) \ar[r]^{R(T)_Z}\ar[d]_\cong&  \K_K(Y\times Z)\ar[d]^\cong \\
\K_G(\Ind_K^G (X\times Z))  & \K_G (\Ind_K^G(Y\times Z)) \\
\K_G (  X\times \Ind_K^G Z) \ar[u]^\cong\ar[r]_{T_{\Ind_K^G Z}} & \K_G (  Y\times \Ind_K^G Z)  \ar[u]_\cong.
}
\end{equation*}
as in (\ref{eq-ind-homo2}).
The Baum-Connes conjecture implies that the maps (\ref{eq-pd-form-bc}) and (\ref{eq-dual-dirac-bc}) induced from   Kasparov's $D\in \K\K_G(TS,\pt)$ and $ D{\check{\phantom{i}}}\in \K\K_G(\pt, TS)$ are isomorphisms.  But $TS$ is $G$-equivariantly homeomorphic to $S\times S$ via the exponential map, and so assuming the Baum-Connes conjecture there are isomorphisms 
\begin{equation}
\label{eq-new-diracs}
D\in \K\K_G(S\times S,\pt)\quad\text{and}\quad  D{\check{\phantom{i}}}\in \K\K_G(\pt, S\times S)
\end{equation}
The latter is given by an almost-equivariant complex of Hermitian vector bundles on $S\times S$, so its action on $\K$-theory commutes with the action of $\KK$-classes (in the sense of Lemma~\ref{lemma-thom-funct}).  Given these things,   the inverse to the restriction map may be defined by means of the diagram
\begin{equation*}
\xymatrix{
\K_G(X\times Z)\ar[r] ^{R^{-1}(T)_Z}\ar[d]_\cong & \K_G(Y\times Z) \ar[d]^\cong\\
\K_G(X\times S\times S \times Z)  & \K_G(Y\times S\times S\times Z) \\
\K_K(X\times S\times Z)\ar[r]_{T_{S\times Z}}\ar[u]^\cong  & \K_K(X\times S\times Z)\ar[u]_\cong
}
\end{equation*}
where the top vertical maps are induced from the class $ D{\check{\phantom{i}}}$ in (\ref{eq-new-diracs}).

Conversely, assume the restriction map is an isomorphism.  The restrictions of $D$ and $ D{\check{\phantom{i}}}$ to $K$-equivariant Kasparov theory  are isomorphisms, and hence they are isomorphisms in in $K$-equivariant $\KK$-theory.  By our assumption they are isomorphisms in $G$-equivariant $\KK$-theory too, which proves that the Baum-Connes assembly map in the form (\ref{eq-pd-form-bc}) is an isomorphism.
\qed
\end{proof}

\subsection{Baum-Connes and the Flag Variety}
\label{sec-bc-flags}
 
 Throughout this section let  $G\subseteq GL(n,\R)$ be a linear connected real reductive Lie group, let $\g$ be its complexified Lie algebra and let $\Flags$ be the flag variety associated to $\g$.  Let $K$ be a maximal compact subgroup in $G$.  
 
 There is a compact form $G_\comp$ of $G$ (that is, a compact connected subgroup of $GL(n,\C)$ with the same complexified Lie algebra $\g$) and we may choose it to contain $K$. Formulate the Weyl character formula for $G_\comp$ in equivariant $\KK$-theory, as in Theorem~\ref{thm-weyl-kk}.  We obtain identities (\ref{eq-weyleq1}) and (\ref{eq-weyleq2}) in $\KK_{G_\comp}$-theory, but using the restriction functor (\ref{eq-kk-resr-map}) from $G_\comp$-equivariant $\KK$-theory to $K$-equivariant $\KK$-theory we obtain the  formulas
 \begin{equation*}
[\pt \leftarrow \Flags] \otimes [\wdg^*T^*\Flags]\otimes [\Flags\to \pt] = |W|\cdot [\pt \to \pt] \in \KK_K(\pt,\pt) 
\end{equation*}
and
\begin{equation*}
[\Flags \to \pt] \otimes [\pt \leftarrow \Flags] \otimes [\wdg^*T^*\Flags] = \sum_{w\in W} [\Flags\stackrel w \to \Flags ]\in \KK_K(\Flags,\Flags) .
\end{equation*}
Here $\Flags$ continues to be the flag variety for $\g$ and $W$ is the Weyl group for $\g$ (not the compact group $K$).
 
 If we assume the Baum-Connes conjecture, then by inverting the restriction functor 
 \[
 \Res_K^G \colon \KK_G(X,Y)\longrightarrow \KK_K(X,Y)
 \]
 we obtain classes $[\Flags\to \pt]$ and so on in $G$-equivariant $\KK$-theory, and the same identities among them as above. In other words, we obtain a $\KK$-theoretic Weyl character formula   for the noncompact group $G$.
 
  Using the notation introduced in Section~\ref{sec-comparison}, we  obtain ``global sections'' and ``localization'' maps 
 \[
 \Gamma\colon \K_G(\Flags\times X)\longrightarrow \K_G(X)
 \quad \text{and}
\quad 
\Lambda \colon  \K_G(X)\longrightarrow \K_G(\Flags\times X) ,
\]
as well as ``intertwining operators''
\[
I_w\colon \K_G (\Flags \times X) \longrightarrow \K_G(\Flags\times X)
\]
with
\[
\Lambda\otimes \Gamma = |W|\cdot \Id \colon \K_G(X) \longrightarrow \K_G(X)
\]
and 
\[
\Gamma\otimes \Lambda = \sum_{w\in W} I_w\colon \K_G( \Flags \times X)\longrightarrow \K_G(\Flags\times X).
\]
 After inverting $|W|$ the   formulas identify $\K_G(X)$ with the $W$-invariant part for $\K_G(\Flags\times X)$, and indeed they imply   Baum and Connes'  conjectural formula (\ref{eqn-bc-formula}) for $\K_G(X)$ (after inverting $|W|$ again).  A more precise (but maybe less useful) statement can be obtained using the Demazure character formula, as in Section~\ref{sec-bott-sam}.

\bibliographystyle{alpha}
\bibliography{References.bib}

\end{document}